\documentclass[a4paper,reqno,12pt]{amsart}

\RequirePackage{doi}
\usepackage{hyperref}

\usepackage{amsmath,amsthm,amssymb}
\usepackage{mathrsfs}
\usepackage[shortlabels]{enumitem}
\usepackage{graphicx}
\usepackage[font=small]{caption}
\usepackage{xcolor}
\usepackage{url}

\newcommand{\N}{\mathbb{N}}
\newcommand{\R}{{\mathbb{R}}}

\newcommand{\C}{{\mathbb{C}}}
\newcommand{\Z}{{\mathbb{Z}}}

\newcommand{\dd}{{{\rm d}}}
\newcommand{\ii}{{\rm i}}

\newcommand{\ov}{\overline}
\newcommand\wt{\widetilde}

\newcommand{\eps}{\varepsilon}

\newcommand{\spp}{\sigma_{\rm p}}

\newcommand{\Dom}{{\operatorname{Dom}}}

\newcommand{\Rank}{{\operatorname{rank}}}
\renewcommand{\Re}{\operatorname{Re}}
\renewcommand{\Im}{\operatorname{Im}}

\newcommand{\BigO}{\mathcal{O}}

\newcommand{\lspan}{{\operatorname{span}}}

\newcommand{\Do}{\frac{\dd}{\dd x}}

\theoremstyle{plain}

\newtheorem{theorem}{Theorem}[section]
\newtheorem{lemma}[theorem]{Lemma}

\theoremstyle{definition}

\newtheorem{asm-sec}[theorem]{Assumption}

\newcommand\cD{\mathcal D}
\newcommand\cE{\mathcal E}

\newcommand\cH{\mathcal H}

\newcommand\cS{\mathcal S}

\usepackage{fullpage}

\begin{document}


\author{Boris Mityagin}

\address[Boris Mityagin]{Department of Mathematics, The Ohio State University, 231 West 18th Ave,
	Columbus, OH 43210, USA}
\email{mityagin.1@osu.edu, boris.mityagin@gmail.com}

\author{Petr Siegl}

\address[Petr Siegl]{Institute of Applied Mathematics, Graz University of Technology, Steyrergasse 30, 8010 Graz, Austria}
\email{siegl@tugraz.at}


\subjclass{34L40, 34L10, 81Q12}
\keywords{Dirac operator, distributional potential, Riesz basis, Paley-Zygmund theorem}

\date{\today}

\title{Eigensystems of Dirac operators with singular potentials}

\begin{abstract}
We perturb one-dimensional Dirac operators on a bounded interval subject to Dirichlet boundary conditions by potentials with Fourier coefficients exhibiting power-decay. As a consequence of Paley-Zygmund theorem, this broad family of potentials comprises distributions that are neither integrable functions nor measures. We localize the spectrum of the perturbed operator and, as the main result, show that its eigensystem generates a Riesz basis.
\end{abstract}

\maketitle

\section{Introduction}

Let $A$ be the Dirac operator in $\cH=L^2(0,\pi) \oplus L^2(0,\pi)$ subject to Dirichlet boundary conditions
\begin{equation}\label{A.def}
\begin{aligned}
A &= 
\ii
\begin{pmatrix}
1 & 0
\\
0 & - 1
\end{pmatrix}
\Do,
\\ 
\Dom(A) & = \{ f = (f_1,f_2) \in H^1(0,\pi) \oplus H^1(0,\pi) \, : \, f_1(0) = f_2(0), f_1(\pi) = f_2(\pi) \}.
\end{aligned}	
\end{equation}
The spectrum of $A$ is discrete, $\sigma(A) = \Z$, all eigenvalues are simple and normalized eigenfunctions can be seleted as
\begin{equation}
\psi_j = \frac{1}{\sqrt {2 \pi}} 
\begin{pmatrix}
e_{-j} \\ e_j
\end{pmatrix},
\qquad e_j(x) = \frac{1}{\sqrt \pi} e^{\ii jx}, \qquad x \in (0,\pi), \quad j \in \Z,
\end{equation}
see e.g.~\cite{Djakov-2010-283}. We denote the linear span of the eigenfunctions by
\begin{equation}\label{E.def}
 \cE = \lspan\{ \psi_j \, : \, j \in \Z\}. 
\end{equation}

The goal of this note is to show that there is a broad family of potentials
\begin{equation}\label{V.def}
V = 
\begin{pmatrix}
	0 & v
	\\
	w & 0
\end{pmatrix},
\end{equation}
where $v$ and $w$ are neither integrable functions nor measures such that the eigensystem of $T = A + V$ generates a Riesz basis (see Theorem~\ref{thm:Dirac.RB}). This appears to be the first result for distributional potentials as such a claim is not covered by the results on Riesz basis properties for perturbations by $L^2$-potentials in \cite{Djakov-2010-283} and by $L^1$-potentials in \cite{Savchuk-2014-96,Lunyov-2016-441}.

Let $v$ and $w$ be defined by their Fourier expansions (in $\cD^*$ introduced below)
\begin{equation}\label{v.w.def}
v = \sum_{m \in 2 \Z} v_m e_m, \qquad w = \sum_{m \in 2\Z} w_m e_m,
\end{equation}
where the sequences $\{v_{2m}\}, \{w_{2m}\} \subset \C$ satisfy 
\begin{equation}\label{v.w.cond.intro}
|v_{2m}| + |w_{2m}| = \BigO \left( |m|^{-\alpha} \right), \quad  |m| \to \infty,
\end{equation}
with some $\alpha>0$.

If $\alpha > 1/2$, then $v$ and $w$ are $L^2$-functions. For $\alpha \in (0,1/2]$, $v$ and $w$ are understood as Schwarz (antilinear) distributions on the Sobolev space $H^{1/2}$, i.e., as bounded  functionals
\begin{equation}
	(v,f) = \sum_{k \in 2\Z} v_k \ov{\langle f, e_k \rangle}, \qquad f \in \cD,
\end{equation}
(and similarly for $w$) on the space
\begin{equation}
\cD:=\{ f \in L^2(0,\pi) \, : \,  \{(|j|+1)^\frac12 \langle f, e_j \rangle \} \in \ell^2(2\Z) \}
\end{equation}
equipped with the norm 
\begin{equation}
\|f\|_{\cD}^2 = \sum_{j \in 2\Z} (|j|+1) |\langle f, e_j \rangle|^2, \quad f \in \cD.
\end{equation}
We view $v$ (and similarly $w$ and $V$) as a multiplication operator in the following sense
\begin{equation}\label{v.mult}
v : \lspan \{e_j \, : \, j \in \Z\} \to \cD^*: 
v e_j = \frac{1}{\sqrt \pi} \sum_{m \in 2 \Z} \tilde v_{m-j} e_m,
\end{equation}
where
\begin{equation}
\wt v_m = 
\begin{cases}
v_m, & \text{ if } m \in 2 \Z,
\\[2mm]
\frac{2\ii}{\pi} \sum_{k \in 2 \Z} \frac{v_k}{m-k}, & \text{ if } m \in 2 \Z+1,
\end{cases}
\end{equation}
see \cite[Lemma 2]{Djakov-2006-61} for details. Since the condition \eqref{v.w.cond.intro} implies that as $|m| \to \infty$
\begin{equation}\label{v.til.est}
|\tilde v_m| = 
\begin{cases}
\BigO \left( \frac{\log |m|}{ |m|^{\alpha}} \right), & 0 < \alpha \leq 1,
\\[2mm]
\BigO \left( \frac{1}{|m|} \right), & \alpha > 1,
\end{cases}
\end{equation}
see \cite[Lemma 2]{Mityagin-2016-106} or \eqref{MS.lem.HT}, \eqref{MS.lem.HT.2} below, $v e_j$, $j \in \Z$, are indeed elements of $\cD^*$.
Notice that \eqref{v.mult} is an extension of the usual multiplication operator by an $L^2$-function if $\alpha>1/2$. 

It follows that $V$ can be viewed as a form defined on $\cE$ and it satisfies
\begin{equation}\label{V.form}
	( V \psi_j, \psi_k )
	= 
	\frac{(v e_j, e_{-k}) + (w e_{-j}, e_k)} 2 
	=
	\frac{\wt v_{-(j+k)}+ \wt w_{j+k} }{2 \sqrt \pi}, \quad j,k \in \Z.
\end{equation}
Hence the ``matrix elements'' of $V$ satisfy
\begin{equation}\label{V.form.tau}
|( V \psi_j, \psi_k )| \leq \tau(j+k), \quad j, k \in \Z,
\end{equation}
where
\begin{equation}
	\tau(j) = \frac{|\wt v_{-j}| + |\wt w_j|}{2 \sqrt \pi}, \quad j \in \Z,
\end{equation}
and, by \eqref{v.til.est}, as $|j| \to \infty,$
\begin{equation}\label{tau.cond}
	\tau(j)
	= \begin{cases}
		\BigO \left( \frac{\log |j|}{ |j|^{\alpha}} \right), & 0 < \alpha \leq 1,
		\\[2mm]
		\BigO \left( \frac{1}{|j|} \right), & \alpha > 1.
		\end{cases}
\end{equation}

The condition on the perturbation $V$ in \eqref{V.form.tau} and \eqref{tau.cond} represents a decay of the matrix elements of $V$ on the diagonals. Such a structure was exploited in the analysis of Dirac operators in \cite{Djakov-2010-283} where it was assumed that $\tau \in \ell^2(\Z)$ (see also \cite{Djakov-2009-481,Djakov-2010-203} for Schr\"odinger operators). We note that \eqref{V.form.tau} and \eqref{tau.cond} is of a different type than the form local-subordination used in \cite{Mityagin-2016-106,Mityagin-2019-139,Mityagin-2025-31,MiSi-2026-I}, which represents a decay of matrix elements of $V$ on rows and columns. Namely, 
\begin{equation}\label{V.loc.sub}
	|(V \psi_j, \psi_k)| \leq \omega_j \omega_k, \quad j, k \in \Z,
\end{equation}
where $\{\omega_j\}$ decays sufficiently fast (e.g.~$\omega_j = \BigO(|j|^{-\alpha})$, $|j| \to \infty$, with $\alpha>0$). 

The Fourier coefficients of $v$ and $w$ are decaying; however, it is important to notice that, as a consequence of Paley-Zygmund theorem (see  \cite[App.~B]{Katznelson-2004}) recalled below, if $\alpha \leq 1/2$, the class determined by \eqref{v.w.def} and \eqref{v.w.cond.intro} contains distributions that are neither integrable functions nor measures. In the statement of Theorem~\ref{thm:PZ}, the system of Rademacher functions $\{{\mathbf r}_n\}$ is a sequence of independent random variables taking the values $1$ and $-1$ with probability $1/2$ for each.
\begin{theorem}[Paley-Zygmund \cite{Paley-1930-26i,Paley-1930-26ii,Paley-1932-28iii}]
	\label{thm:PZ}
Let $\{a_n\}_{n \in 2\Z} \subset \C$ satisfy
\begin{equation}
	\sum_{n \in 2\Z} |a_n|^2 = \infty.
\end{equation}
Then the series
\begin{equation}\label{sum.rad}
\sum_{n \in 2\Z} a_n {\mathbf r}_n e^{\ii n x}
\end{equation}
is almost surely \emph{neither} Fourier series of an $L^1$-function nor of a Borel measure on $(0,\pi)$. 
\end{theorem}

To state our result, we introduce the following objects. For $N \in \N$ and $h>0$, let 
\begin{equation}\label{Pi.def}
	\begin{aligned}
		\Sigma_0 & \equiv \Sigma_0(N,h) := \left\{z \in \C \, : \,  |\Re z| \leq N + \frac 12, \ |\Im z| \leq h \right\}, 
		\\ \Pi_k&:= B_\frac12(k), \quad 
		\Pi:=\bigcup_{k \in \N} \Pi_k, \quad k \in \Z.
	\end{aligned}
\end{equation}
Let $S_0^0$ and $P_k^0$, $k \in \Z$, be the Riesz projections of the unperturbed operator $A$ corresponding to the eigenvalues lying in $\Sigma_0$ and $\Pi_k$, respectively, 
\begin{equation}\label{Pn.SN.0.def}
	\begin{aligned}
		S_0^0:=\frac{1}{2\pi \ii} \int_{\partial \Sigma_0}(z-A)^{-1} \dd z, \quad 
		P_k^0  := \frac{1}{2\pi \ii} \int_{\partial \Pi_k}(z-A)^{-1} \dd z, \quad k \in \Z.
	\end{aligned}
\end{equation}

\begin{theorem}\label{thm:Dirac.RB}
Let $A$ be as in \eqref{A.def} and let $V$ be of the form \eqref{V.def}, \eqref{v.w.def}. Suppose that for some $\alpha>0$ the coefficients of $v$ and $w$ satisfy
\begin{equation}\label{v.w.cond}
|v_{2m}| + |w_{2m}| = \BigO \left(|m|^{-\alpha} \right), \quad |m| \to \infty. 
\end{equation}
Let $T = A+V$ be defined as in Subsection~\ref{ssec:T.def}. Then
\begin{enumerate}[\upshape (i)]
\item $T$ has compact resolvent and there exist $N_0 \in \N$ and $h_0>0$ such that 
	\begin{equation}\label{T.spec.loc}
	\sigma(T) = \spp(T) \subset \Sigma_0(N_0,h_0) \cup \bigcup_{|k| > N_0} \Pi_k;
	\end{equation}
\item with these $N_0$ and $h_0$, Riesz projections 
\begin{equation}\label{Pn.SN.def}
	\begin{aligned}
		S_0:=\frac{1}{2\pi \ii} \int_{\partial \Sigma_0}(z-T)^{-1} \dd z, \quad 
		P_n  := \frac{1}{2\pi \ii} \int_{\partial \Pi_n}(z-T)^{-1} \dd z, \quad |n|> N_0,
	\end{aligned}
\end{equation}
are well-defined, disjoint and  
\begin{equation}\label{rank.SN.Pn}
		\Rank  S_0  = \Rank \, S_0^0 = 2 N_0+1, 
		\quad
		\Rank  P_n  = \Rank \,  P_n^0 = 1, \quad  |n|>N_0;
\end{equation}
\item the system of projections
\begin{equation}\label{S0Pk.syst}
	\{S_0\} \cup \{P_n\}_{|n|>N_0}	
\end{equation}
is complete;
\item the system of projections \eqref{S0Pk.syst} has the Riesz property (hence the eigensystem of $T$ generates a Riesz basis of $L^2(0,\pi) \oplus L^2(0,\pi)$).	
\end{enumerate}
\end{theorem}

The estimates in Section~\ref{sec:proof} below suggest that Theorem~\ref{thm:Dirac.RB} could be proved for $\{v_{2m}\}$, $\{w_{2m}\} \in \ell^{2p}(\Z)$ with $p \in [1,\infty)$, which is a weaker condition than \eqref{v.w.cond}; the case $p=1$ is analyzed in \cite{Djakov-2010-283}. However, in the last part of the proof in Subsection~\ref{ssec:RB}, this would lead to more extensive estimates (similar to those in \cite{Djakov-2010-283,Mityagin-2025-31}).   

\section{Proof of Theorem~\ref{thm:Dirac.RB}}
\label{sec:proof}

Let 
\begin{equation}\label{K.def}
	K \equiv K(z):= (z-A)^{-\frac 12}, \quad z \in \rho(A),
\end{equation}
where for $0\neq w \in \C$ and $s\in \R$, the $s$-power of $w$ is taken as $w^s := |w|^s e^{\ii s \arg w}$ with $-\pi < \arg w \leq \pi$. 
By \eqref{V.form}, the sesquilinear form
\begin{equation}\label{KVK.form.def}
(V K \cdot, K^* \cdot ): \cE \times \cE \to \C
\end{equation}
satisfies 
\begin{equation}
(V K \psi_j, K^* \psi_k ) = \frac 12 \frac{\wt v_{-(j+k)}+ \wt w_{j+k}}{(z-j)^\frac12 (z-k)^\frac12}, \quad j,k \in \Z, \ z \in \rho(A).
\end{equation}
It follows that for any $ f\in \cE$
\begin{equation}\label{KVK.form.est}
|(V K f, K^* f )|^2 \leq \sum_{k,m \in \Z} \frac{|\tau(k+m)|^2}{|z-k||z-m|} \|f\|^2.
\end{equation}
The estimates in Lemma~\ref{lem:tau.est} below (see also the proof of Lemma~\ref{lem:KVK.norm}) show that the sum on the r.h.s.~of \eqref{KVK.form.est} is finite for any $z \in \rho(A)$. Hence the form \eqref{KVK.form.def} defines a bounded operator $B \equiv B(z)$ on $\cH$ by
\begin{equation}\label{B.def}
	\langle B(z) f,  g \rangle = (V K f, K^* g ), \quad f, g \in \cE, \ z \in \rho(A)
\end{equation}
and from \eqref{KVK.form.est}, we have
\begin{equation}\label{KVK.HS}
\|B(z)\|_{\rm HS}^2 \leq \sum_{k,m \in \Z} \frac{|\tau(k+m)|^2}{|z-k||z-m|}, \quad z \in \rho(A),
\end{equation}

\subsection{Estimates of $\|B(z)\|$}

We first prove the following technical lemma.

\begin{lemma}\label{lem:tau.est}
	Let $\tau \in \ell^{2p}(\Z)$ with $1 \leq p < \infty$. Then for every $\delta \in (0,2)$
	\begin{align}
		\sum_{k \neq n} \frac{|\tau(n+k)|^2}{|n-k|} & \leq 
		\frac{4p}{\delta} \|\tau\|_{\ell^{2p}}^\delta  \sup_{|j|\geq |n|} |\tau(j)|^{2-\delta} +  2p \frac{\|\tau\|_{\ell^{2p}}^2}{|n|^{\frac1p}}, \label{r.sum.1}
		\\
		\sum_{j,k \neq n} \frac{|\tau(j+k)|^2}{|n-j||n-k|}  & \leq 
	 \frac{64 p^2}{\delta^2} 	\|\tau\|_{\ell^{2p}}^\delta \sup_{|j|\geq |n|} |\tau(j)|^{2-\delta} +  128 p^2 \frac{\|\tau\|_{\ell^{2p}}^2}{|n|^{ \frac1{2p}}},
		\quad |n|>2.
		\label{r.sum.2}
	\end{align}
\end{lemma}
\begin{proof}
	We proceed as in \cite[Lem.~4.1]{Djakov-2010-283}. We give details to the case $p>1$; the case $p=1$ is analogous. Let $n \in \Z$ with $|n|>2$. 
	
	We split the sum on the l.h.s.~of \eqref{r.sum.1} as
	\begin{equation}\label{r.sum.1.a}
		\sum_{k \neq n}^\infty \frac{|\tau(n+k)|^2}{|n-k|} = \sum_{0<|n-k| \leq |n|} \frac{|\tau(n+k)|^2}{|n-k|} + \sum_{|n-k| > |n|} \frac{|\tau(n+k)|^2}{|n-k|}.
	\end{equation}
	First, notice that if $|n-k| \leq |n|$, then $|n+k| \geq 2|n|-|n-k| \geq |n|$. Hence
	\begin{equation}
		\sum_{0<|n-k| \leq |n|} \frac{|\tau(n+k)|^2}{|n-k|} 
		\leq 
		\sup_{|j|\geq |n|} |\tau(j)|^{2-\delta}  \sum_{0<|n-k| \leq |n|} \frac {|\tau(n+k)|^\delta }{|n-k|}.
	\end{equation}
	By H\"older inequality with $r=2p/\delta \in (1,\infty)$ and the conjugated $s$, we obtain
	\begin{equation}
	\sum_{0<|n-k| \leq |n|} \frac {|\tau(n+k)|^\delta }{|n-k|} 
	\leq
	\|\tau\|_{2p}^\delta 
	\left( 2 \sum_{k =1}^\infty \frac{1}{k^s} \right)^\frac 1s.
	\end{equation}
	Using
	\begin{equation}\label{sum.s.est}
	\left( 2 \sum_{k =1}^\infty \frac{1}{k^s} \right)^\frac 1s \leq \left(2 + 2 \int_1^\infty \frac{\dd x}{x^s} \right)^\frac1s = (2r)^\frac1s \leq 2r = \frac{4p}{\delta}, 
	\end{equation}	
	we arrive at
	\begin{equation}
	\sum_{0<|n-k| \leq |n|} \frac{|\tau(n+k)|^2}{|n-k|} 
	\leq 
	\frac{4p}{\delta} \|\tau\|_{2p}^\delta  \sup_{|j|\geq |n|} |\tau(j)|^{2-\delta}.
	\end{equation}
	For the second sum on the r.h.s.~of \eqref{r.sum.1.a}, H\"older inequality yields
	\begin{equation}
		\begin{aligned}
			\sum_{|n-k| > |n|} \frac{|\tau(n+k)|^2}{|n-k|} & \leq
			\left(
			\sum_{|n-k| > |n|} |\tau(n+k)|^{2p}
			\right)^{\frac 1p}
			\left(
			\sum_{|n-k| > |n|} \frac{1}{|n-k|^q}
			\right)^\frac1q
			\\
			& \leq \|\tau\|_{\ell^{2p}}^2 
			\left(
			2 \sum_{k > |n|} \frac{1}{k^q}
			\right)^\frac1q 
			\leq 
			\|\tau\|_{\ell^{2p}}^2 
			\left(
			2 \int_{|n|}^\infty \frac{\dd x}{x^q}
						\right)^\frac1q  
			\leq 
			2p
			\frac{\|\tau\|_{\ell^{2p}}^2}{|n|^\frac1p}.
		\end{aligned}
	\end{equation}
In summary, \eqref{r.sum.1} is proved.

	To show \eqref{r.sum.2} we split the summation to three subsets
	\begin{equation}
		\begin{aligned}
			J_1 &:= \{ (j,k) \, : \, 0<|n-j| \leq |n|/2, \ 0<|n-k| \leq |n|/2\},
			\\
			J_2 &:= \{ (j,k) \, : \, j \neq n, \ |n-k| > |n|/2\},
			\\
			J_3 &:= \{ (j,k) \, : \, |n-j| > |n|/2, \ k \neq n\}.
		\end{aligned}
	\end{equation}
	For the first subset $J_1$, notice that
	\begin{equation}
		|j+k| = |2n - (n-j) - (n-k)| \geq 2 |n| - |n-j| - |n-k| \geq |n|, 
	\end{equation}
	hence,
	\begin{equation}\label{sum.db.1}
		\sum_{(j,k) \in J_1}  \frac{|\tau(j+k)|^2}{|n-j||n-k|} 
		\leq
		\sup_{|j|\geq |n|} |\tau(j)|^{2-\delta} \sum_{(j,k) \in J_1}  \frac{|\tau(j+k)|^\delta}{|n-j||n-k|}.
	\end{equation}
	For any $\beta$ satisfying
	\begin{equation}\label{beta.cond.del}
	  \frac12 \left( 1 - \frac{\delta}{2p}  \right) <\beta < \frac 12,
	\end{equation}
	we obtain by Cauchy-Schwarz inequality that
	\begin{equation}
	\begin{aligned}
	& \sum_{(j,k) \in J_1}  \frac{|\tau(j+k)|^\delta}{|n-j||n-k|} 
	 =
	\sum_{(j,k) \in J_1}  \frac{|\tau(j+k)|^\delta }{|n-j|^{\beta+(1-\beta)}|n-k|^{(1-\beta)+\beta}} 
	\\
	& \qquad 
	\leq 
	\left(\sum_{(j,k) \in J_1}  \frac{|\tau(j+k)|^\delta}{|n-j|^{2 \beta}|n-k|^{2-2\beta}} \right)^\frac12
	\left(\sum_{(j,k) \in J_1}  \frac{|\tau(j+k)|^\delta}{|n-j|^{2 -2\beta}|n-k|^{2\beta}} \right)^\frac12
	\\
	& \qquad 
	\leq 
	\sum_{j,k \neq n}  \frac{|\tau(j+k)|^{\delta}}{|n-j|^{2 \beta}|n-k|^{2-2\beta}}
	= 
	\sum_{k \neq n} \frac{1}{|n-k|^{2-2\beta}} \sum_{j \neq n} \frac{|\tau(j+k)|^{\delta}}{|n-j|^{2 \beta}}.
	\end{aligned}	
	\end{equation}
	Applying H\"older inequality with $r = 2 p/\delta \in (1,\infty)$ and the conjugated $s$, we obtain
	\begin{equation}\label{sum.db.2}
	\begin{aligned}
			 \sum_{(j,k) \in J_1}  \frac{|\tau(j+k)|^\delta}{|n-j||n-k|} 
			& \leq 
			\|\tau\|_{\ell^{2p}}^\delta  \sum_{k = 1}^\infty \frac{2}{k^{2-2\beta}} 
		\left(
		2 \sum_{j =1}^\infty \frac{1}{j^{2 \beta s}}
		\right)^\frac{2 \beta}{2 \beta s}.
		\end{aligned}	
	\end{equation}
Employing \eqref{sum.s.est}, we have
\begin{equation}
\left(
2 \sum_{j =1}^\infty \frac{1}{j^{2 \beta s}}
\right)^\frac{2 \beta}{2 \beta s}
\leq 
\left(
\frac{4 \beta }{2 \beta -\frac1s}
\right)^\frac1s
\leq 
\left(
\frac{ 2 }{2 \beta -\frac1s}
\right)^\frac1s.		
\end{equation}
We select
\begin{equation}
\beta = \frac1{s+1};
\end{equation}	
notice that \eqref{beta.cond.del} is satisfied and also
\begin{equation}
\frac1 {2 \beta - \frac 1s } = \frac{s+1}{1-\frac 1s} = r (s+1).
\end{equation}	
Hence
\begin{equation}\label{sum.db.3}
\left(
2 \sum_{j =1}^\infty \frac{1}{j^{2 \beta s}}
\right)^\frac{2 \beta}{2 \beta s}
\leq 
(2 r)^\frac1s (s+1)^\frac1s \leq 4 r;
\end{equation}
in the last step we use that $(s+1)^\frac1s$ is decreasing on $[1,\infty)$ and equal $2$ if $s=1$. 
Similarly, we arrive at
\begin{equation}\label{sum.db.4}
\sum_{k = 1}^\infty \frac{2}{k^{2-2\beta}} \leq 4r.
\end{equation}
Combining \eqref{sum.db.1}, \eqref{sum.db.2}, \eqref{sum.db.3} and \eqref{sum.db.4}, we obtain
\begin{equation}\label{sum.db.tot}
	\sum_{(j,k) \in J_1}  \frac{|\tau(j+k)|^2}{|n-j||n-k|} 
	\leq
	16 r^2 	\|\tau\|_{\ell^{2p}}^\delta \sup_{|j|\geq |n|} |\tau(j)|^{2-\delta} = \frac{64 p^2}{\delta^2} 	\|\tau\|_{\ell^{2p}}^\delta \sup_{|j|\geq |n|} |\tau(j)|^{2-\delta}.
	\end{equation}

	For the remaining sums over $J_2$ and $J_3$, Cauchy-Schwarz inequality yields
	\begin{equation}\label{r.2.sum}
		\begin{aligned}
			&\sum_{(j,k) \in J_2}  \frac{|\tau(j+k)|^2}{|n-j||n-k|} = 
			\sum_{(j,k) \in J_3}  \frac{|\tau(j+k)|^2}{|n-j||n-k|} 
			\\
			&\qquad \leq 
			\left(\sum_{(j,k) \in J_3}  \frac{|\tau(j+k)|^2}{|n-j|^{2 \beta}|n-k|^{2-2\beta}} \right)^\frac12
			\left(\sum_{(j,k) \in J_3}  \frac{|\tau(j+k)|^2}{|n-j|^{2 -2\beta}|n-k|^{2\beta}} \right)^\frac12,
		\end{aligned}
	\end{equation}
	where we choose $\beta$ so that 
	\begin{equation}\label{beta.cond}
	\frac{1}{2} \left(1 - \frac 1p\right) = \frac 1 {2q} < \beta < \frac 12.
	\end{equation}
	For the next steps we introduce a parameter
	\begin{equation}\label{t.ineq}
	t = 1 - 2\beta \in (0, 1/p) \subset (0,1],
	\end{equation}
	which we choose as
	\begin{equation}
	t = \frac{q-1}{q+1};
	\end{equation}
	notice that \eqref{beta.cond} holds and also
	\begin{equation}\label{t.id}
	q \left(\frac 1p - t\right) = (q-1) \left( 1 - \frac{q}{q+1}	\right)  = t.
	\end{equation}
	
	By H\"older inequality, \eqref{t.id}, \eqref{t.ineq} and $1/t \leq 2p$ in the last step,
	\begin{equation}\label{r.2.sum.a}
		\begin{aligned}
			&\sum_{(j,k) \in J_3}  \frac{|\tau(j+k)|^2}{|n-j|^{2\beta} |n-k|^{1+t} }
			= \sum_{k \neq n} \frac{1}{|n-k|^{1+t}} \sum_{|n-j|> |n|/2} \frac{|\tau(j+k)|^2}{|n-j|^{2 \beta}} 
			\\
			& \quad 
			 \leq 
			\|\tau \|_{\ell^{2p}}^2
			\sum_{k \neq n} \frac{ 1 }{|n-k|^{1+t}} 
			\left(
			\sum_{|n-j|> |n|/2} \frac{1}{|n-j|^{2 \beta q}}
			\right)^\frac1q
			\\
			& \quad 
			 \leq
			2^{1+\frac1q} \|\tau \|_{\ell^{2p}}^2 \sum_{k =1 }^\infty \frac{ 1 }{k^{1+t}} 
			\left(
			\sum_{j> |n|/2} \frac{1}{j^{2 \beta q}}
			\right)^\frac1q
			\\
			& \quad   
			\leq 
			2^{1+\frac1q} \|\tau \|_{\ell^{2p}}^2 \left(1 + \int_1^\infty \frac{\dd x}{x^{1+t}}\right)
			\left( \int_{|n|/2}^\infty \frac{\dd x}{x^{2 \beta q}} \right)^\frac1q
			\\
			& \quad 
			= 2^{1+2 \beta } \frac{1+t}{t} \frac{1}{\left(q \left(\frac 1p-t \right) \right)^\frac1q} \frac{\|\tau \|_{\ell^{2p}}^2}{|n|^{\frac1p-t}}
			\leq  2^5 p^2 \frac{\|\tau \|_{\ell^{2p}}^2}{|n|^{\frac1p-t}}.
		\end{aligned}
	\end{equation}
	Similarly, using that $1/t \leq 2p$ and $q^{1/q} \leq 2$ in the last step,
	\begin{equation}\label{r.2.sum.b}
		\begin{aligned}
			&\sum_{(j,k) \in J_3}  \frac{|\tau(j+k)|^2}{|n-j|^{1+t} |n-k|^{2\beta} }
			= \sum_{|n-j| > |n|/2 } \frac{1}{|n-j|^{1+t}} \sum_{k \neq n} \frac{|\tau(j+k)|^2}{|n-k|^{2\beta}}   
			\\
			& \quad \leq 
			2^{1+\frac 1q} \|\tau\|_{\ell^{2p}}^2 \sum_{j > |n|/2 } \frac{1 }{j^{1+t}} 
			\left(
			\sum_{k =1 }^\infty \frac{1}{k^{2 \beta q}}
			\right)^\frac1q
			\leq 
			\frac{2^3 q^\frac1q}{t \left(q \left(\frac 1p-t \right) \right)^\frac1q} \frac{\|\tau\|_{\ell^{2p}}^2}{|n|^t}
			\\
			& \quad 
			\leq 2^6 p^2  \frac{\|\tau\|_{\ell^{2p}}^2}{|n|^t}.
		\end{aligned}
	\end{equation}
Combining \eqref{r.2.sum}, \eqref{r.2.sum.a} and \eqref{r.2.sum.b}, we arrive at
	\begin{equation}
		\sum_{(j,k) \in J_2}  \frac{|\tau (j+k)|^2}{|n-j||n-k|} + \sum_{(j,k) \in J_3}  \frac{|\tau (j+k)|^2}{|n-j||n-k|}
		\leq 
		 128 p^2  \frac{\|\tau \|_{\ell^{2p}}^2 } {|n|^{\frac1{2p}}}.
	\end{equation}
thus \eqref{r.sum.2} is proved.
\end{proof}

For $N \in \N$, let 
	\begin{equation}\label{sigma.N.def}
	\sigma_N:= \sup_{|n| \geq N} \left[
	|\tau(2n)|^2 + \sum_{j \neq n} \frac{|\tau(j+n)|^2}{|n-j|} + \sum_{j,k \neq n} \frac{|\tau(j+k)|^2}{|n-j||n-k|}
	\right]^\frac12; 
\end{equation}	
notice that for $\tau$ satisfying \eqref{tau.cond}, we get by Lemma~\ref{lem:tau.est} that for any fixed $\eps>0$
\begin{equation}\label{sig.dec}
\sigma_N = 
\begin{cases}
\BigO(N^{-1/4 }) ,  & \text{if } \alpha > 1/2,
\\
\BigO(N^{-\alpha/2 +\eps}), & \text{if } \alpha \in (0,1/2],
\end{cases}
\quad 	N \to +\infty.
\end{equation}

Next we estimate the norm of $B(z)$. 

\begin{lemma}\label{lem:KVK.norm}
Let $V$ be as in \eqref{V.def}, \eqref{v.w.def}, \eqref{v.w.cond.intro}, let $K$ be as in \eqref{K.def}, let $B$ be the bounded operator determined by \eqref{B.def}, and let $\sigma_N$ be as in \eqref{sigma.N.def}. Then
\begin{enumerate}[\upshape (i)]
	\item we have
	\begin{equation}\label{KVK.est.N}
		\sup_{\substack{z \notin \Pi \\[1mm] |\Re z| \geq N - \frac 12} } \|B(z)\| 
		\leq 4 \sigma_N = o(1), \quad N \to +\infty;
	\end{equation}	 
\item for any fixed $N_0 \in \N$, 
\begin{equation}\label{KVK.est.Y}
\sup_{\substack{|\Re z| \leq N_0 \\[1mm] |\Im z| \geq Y }} \|B(z)\| = o(1), \quad Y \to + \infty.
\end{equation}
\end{enumerate}
\end{lemma}

\begin{proof}
Recall the estimate of the norm of $B(z)$  in \eqref{KVK.HS}.
\begin{enumerate}[\upshape (i), wide]
\item 
	Let %
	\begin{equation}\label{Sigma.def}
		\Xi_n:= \left\{ z \in \C \, : \, n - \frac 12 \leq \Re z \leq n + \frac 12 \right\}, \quad n \in \Z, 
	\end{equation}
	i.e.,~$\Xi_n$ are vertical strips around the eigenvalues $n \in \Z$. Let $z \in \Xi_n$ for some $n \in \Z$ with $z \notin \Pi$. Since $|z-n| \geq 1/2$,
	\begin{equation}\label{sum.KVK.z.1}
		\sum_{j,k} \frac{|\tau(j+k)|^2}{|z-j||z-k|} \leq  4 |\tau(2n)|^2 + 4 \sum_{j \neq n } \frac{|\tau(j+n)|^2}{|z-j|}
		+ \sum_{j,k \neq n} \frac{|\tau(j+k)|^2}{|z-j||z-k|}.
	\end{equation} 
	Next, for $j \neq n$, using that $|\Re z -n | \leq 1/2$,
	\begin{equation}\label{sum.KVK.z.2}
		\begin{aligned}
			|z-j| &\geq |\Re z -j| \geq |n-j|-|\Re z -n|
			\\
			& \geq \frac 12 |n-j| + \frac 12 (|n-j|- 2|\Re z -n| ) 
			\geq \frac 12(|n-j| + |n-j| -1)  
			\\
			& \geq \frac 12 |n-j|.
		\end{aligned}
	\end{equation}
	Hence, for any $z \in \Xi_n$ with $z \notin \Pi$, 
	\begin{equation}\label{sum.KVK.z.3}
		\sum_{j,k} \frac{|\tau(j+k)|^2}{|z-j||z-k|} \leq  4 |\tau(2n)|^2 + 8 \sum_{j \neq n } \frac{|\tau(j+n)|^2}{|n-j|}
		+ 4 \sum_{j,k \neq n} \frac{|\tau(j+k)|^2}{|n-j||n-k|}
	\end{equation} 
	and so \eqref{KVK.est.N} is proved.
\item Let $|\Re z| \leq N_0$ with the fixed $N_0 \in \N$ and let  $|\Im z | > Y$ with $Y>0$.
	Then
	\begin{equation}\label{sum.split}
		\begin{aligned}
			\sum_{j,k} \frac{|\tau(j+k)|^2}{|z-j||z-k|} & \leq 
			\sum_{|j| < 2N_0 ,k < 2 N_0} \frac{|\tau(j+k)|^2}{|z-j||z-k|} + \sum_{|j| \geq 2N_0, |k| \geq 2N_0} \frac{|\tau(j+k)|^2}{|z-j||z-k|} 
			\\ & \qquad + 2 \sum_{|j| < 2N_0, |k| \geq 2N_0} \frac{|\tau(j+k)|^2}{|z-j||z-k|}  .
		\end{aligned}
	\end{equation} 
	First, the finite sum can be estimated in a straightforward way as
	\begin{equation}\label{sum.split.1}
		\sum_{|j| < 2N_0 ,k < 2 N_0} \frac{|\tau(j+k)|^2}{|z-j||z-k|}
		\leq 
		\frac{1}{Y^2} \sum_{|j| < 2N_0 ,k < 2 N_0} |\tau(j+k)|^2 = o(1), \quad Y \to +\infty.
	\end{equation}
	To estimate the second sum, notice that for $|j| \geq  2N_0$
	\begin{equation}
		|\Re z - j| \geq |j| - |\Re z| \geq \frac {|j|}2  + N_0 - N_0 \geq \frac {j}2. 
	\end{equation}
	Thus
	\begin{equation}
		\begin{aligned}
			\sum_{|j| \geq 2N_0, |k| \geq 2N_0} \frac{|\tau(j+k)|^2}{|z-j||z-k|} 
			&\leq 4 \sum_{|j| \geq 2N_0, |k| \geq 2N_0} \frac{|\tau(j+k)|^2}{(|\Re z-j|+Y)(|\Re z-k|+Y)} 
			\\
			& \leq 16 \sum_{|j| \geq 2N_0, |k| \geq 2N_0} \frac{|\tau(j+k)|^2}{(|j|+Y)(|k|+Y)}. 
		\end{aligned}	
	\end{equation}
	Let $p> \max\{1, 1/(2\alpha) \}$ where $\alpha$ is as in \eqref{tau.cond} and let $ \gamma$ be such that
	\begin{equation}\label{gam.sel.0}
		\frac12 \left(1- \frac{1}{p} \right) < \gamma < \frac 12.
	\end{equation}
	Then Cauchy-Schwarz inequality yields 
	\begin{equation}
		\begin{aligned}
			\sum_{|j| \geq 2N_0, |k| \geq 2N_0} \frac{|\tau(j+k)|^2}{(|j|+Y)(|k|+Y)}	& \leq  \sum_{j,k \neq 0} \frac{|\tau(j+k)|^2}{|j|^{\gamma+(1-\gamma)}|k|^{(1-\gamma)+\gamma}}
			\\
			& \leq 
			\left(\sum_{j,k \neq 0 }  \frac{|\tau(j+k)|^2}{|j|^{2 \gamma}|k|^{2-2\gamma}} \right)^\frac12
			\left(\sum_{j,k \neq 0}  \frac{|\tau(j+k)|^2}{|j|^{2 -2\gamma}|k|^{2\gamma}} \right)^\frac12
			\\
			& = \sum_{k \neq 0 } \frac1 {|k|^{2-2\gamma}}  \sum_{j \neq 0 } \frac{|\tau(j+k)|^2}{|j|^{2 \gamma}}.
		\end{aligned}	
	\end{equation}
	Next, by H\"older inequality
	\begin{equation}
		\sum_{j \neq 0 } \frac{|\tau(j+k)|^2}{|j|^{2 \gamma}}	
		\leq 
		\|\tau\|_{\ell^{2p}}^2 
		\left(
		\sum_{j \neq 0} \frac{1}{|j|^{2 \gamma q}}
		\right)^\frac1 q.
	\end{equation}
	The choice of $\gamma$ in \eqref{gam.sel.0} guarantees that both $2-2\gamma>1$ and  $2 \gamma q >1$, so there exists $C>0$, independent of $Y$ such that
	\begin{equation}
	\sum_{|j| \geq 2N_0, |k| \geq 2N_0} \frac{|\tau(j+k)|^2}{(|j|+Y)(|k|+Y)} \leq C \|\tau\|_{\ell^{2p}}^2 < \infty.	
	\end{equation}	
	Hence the dominated convergence theorem yields that
	\begin{equation}\label{sum.split.2}
		\sum_{|j| \geq 2N_0, |k| \geq 2N_0} \frac{|\tau(j+k)|}{(|j|+Y)(|k|+Y)}  = o(1), \quad Y \to + \infty.
	\end{equation}
	The estimate of the third sum on the r.h.s.~of \eqref{sum.split} is simpler; namely by H\"older inequality (with $p> \max\{1, 1/(2\alpha) \}$)
	\begin{equation}\label{sum.split.3}
		\begin{aligned}
			\sum_{|j| < 2N_0, |k| \geq 2N_0} \frac{|\tau(j+k)|^2}{|z-j||z-k|} 
			& \leq 
			\frac{4}{Y} \sum_{|j| < 2N_0, |k| \geq 2N_0}  \frac{|\tau(j+k)|^2}{|k|} 
			\\
			& \leq 
			\frac{8 N_0}{Y} \|\tau\|_{\ell^{2p}}^2 \left(\sum_{k \neq 0} \frac{1}{|k|^q}\right)^{\frac1q} = o(1), \quad Y \to + \infty.
		\end{aligned}
	\end{equation}
	Combining \eqref{sum.split}, \eqref{sum.split.1}, \eqref{sum.split.2} and \eqref{sum.split.3}, we arrive at \eqref{KVK.est.Y}. \qedhere
\end{enumerate}
\end{proof}

\subsection{Definition of $T$, spectral localization and completeness}
\label{ssec:T.def}

It follows from \eqref{KVK.est.Y} that there exists $t_0>0$ such that 
\begin{equation}
\|B(\ii t_0)\| \leq \frac 12,
\end{equation}
thus a densely defined closed operator $T$ with non-empty resolvent set, representing the form sum of $A$ and $V$, can be defined similarly to the pseudo-Friedrichs extension, see e.g.~\cite[Thm.~VI.3.11]{Kato-1966},  \cite[Thm.~IV.4.2]{EE} or \cite[Sec.~2]{Shkalikov-2016-71}. Moreover, for $z \in \rho(A)$ such that $I - B(z)$ is boundedly invertible the resolvent of $T$ has a representation 
\begin{equation}\label{T.res.repr}
(z-T)^{-1} = K(z)(I - B(z))^{-1} K(z).
\end{equation}
Since $K$ is in the Schatten class $\cS_{2p}$ for any $p>1$, the resolvent of $T$ is in the Schatten class $\cS_p$ for any $p>1$.

From \eqref{KVK.est.N}, there exists $N_0 \in \N$ such that 
\begin{equation}\label{N_0.def}
	\sup_{\substack{z \notin \Pi \\[1mm] |\Re z| \geq N_0 - \frac 12} } \|B(z)\| 
	\leq \frac 12.
\end{equation}	
and, for this $N_0$, by \eqref{KVK.est.Y} there exists $h_0>0$ such that 
\begin{equation}\label{h.def}
	\sup_{\substack{|\Re z| \leq N_0  \\[1mm] |\Im z| \geq h_0 }} \|B(z)\| \leq \frac 12.
\end{equation}
Employing \eqref{T.res.repr}, we arrive at the spectral localization for $T$ in \eqref{T.spec.loc}. Moreover, the Riesz projections in \eqref{Pn.SN.def} are well-defined, disjoint and the claim on the ranks in \eqref{rank.SN.Pn} follows by a standard perturbation argument (see e.g.~\cite[proof of Prop.~3.3]{Mityagin-2025-31}).

Finally, the completeness of \eqref{S0Pk.syst} follows by \cite[Cor.XI.9.31]{DS2} since $(\ii t_0 - T)^{-1} \in \cS_p$ with any $p>1$ and for any $\omega \in (-\pi,0) \cup (0,\pi)$ 
\begin{equation}
	\|(r e^{\ii \omega} - T)^{-1}\| = \BigO(1/r), \quad r \to +\infty.
\end{equation}

\subsection{Riesz property}
\label{ssec:RB}
We employ the criterium \cite[Chap.~6]{Gohberg-1969}, \cite[\S 6]{Shkalikov-2016-71} or \cite[Thm.~A]{Motovilov-2017-8}. Since the system \eqref{S0Pk.syst} is disjoint and complete, we need to show that
\begin{equation}\label{RB.sum.crit}
\forall f \in \cH: \	\sum_{|n|\geq N_0} |\langle P_n f,f \rangle| < \infty.
\end{equation}
For $z \in \partial \Pi_n$ with $|n| \geq N_0$, we expand the resolvent of $T$ in \eqref{T.res.repr} further and obtain
	\begin{equation}\label{res.dif}
		(z-T)^{-1} = (z-A)^{-1} + K(z) B(z) (I-B(z))^{-1}K(z).
	\end{equation}  
	Since $A=A^*$, to verify \eqref{RB.sum.crit}, it suffices to show that for every $f \in \cH$
	\begin{equation}\label{sum.conv}
		\sum_{|n|>N_0}  \int_{\partial \Pi_n} | \langle B(z) (I-B(z))^{-1} K(z)f, K(z)^* f \rangle | |\dd z| < \infty.
	\end{equation}

	We have
	\begin{equation}\label{Kf.norm}
		\|K(z)f\|^2 = \sum_{k \in \Z} |\langle K(z)f, \psi_k \rangle|^2 = \sum_{k \in \Z} \frac{|f_k|^2}{|z-k|} = \|K(z)^* f\|^2, 
	\end{equation}
	where we decomposed $f$ in the orthonormal basis $\{\psi_k\}$, i.e., $f = \sum_k f_k \psi_k$.
	Using that $\|B(z)\| \leq 1/2$ for $z \in \partial \Pi_n$ with $|n| \geq N_0$, we obtain that the l.h.s.~of \eqref{sum.conv} does not exceed
	\begin{equation}\label{sum.conv.2}
	\begin{aligned}
		\sum_{|n|>N_0}  \int_{\partial \Pi_n}  \frac{\|B(z)\|}{1-\|B(z)\|} \sum_{k \in \Z} \frac{|f_k|^2}{|z-k|} |\dd z|
		& \leq 
		2 \pi \sum_{k \in \Z} 
		\left(
		\sum_{|n| >N_0} \sup_{z \in \partial \Pi_n} \frac{\|B(z)\|}{|z -k|} 
		\right)
		|f_k|^2,
		\\
		& \leq 
		2 \pi \sup_{k \in \Z} \sum_{|n|>N_0} \sup_{z \in \partial \Pi_n} \frac{\|B(z)\|}{|z-k|} \|f\|^2.
	\end{aligned}	
	\end{equation}
Recall that by \eqref{KVK.est.N}
\begin{equation}
	\sup_{z \in \partial \Pi_n}\|B(z)\| 
	\leq 4 \sigma_{|n|}, \quad |n|>N_0,
\end{equation}
thus, similarly as in \eqref{sum.KVK.z.2}, 
\begin{equation}\label{KVK.cond.Dirac}
\sum_{|n|>N_0} \sup_{z \in \partial \Pi_n} \frac{\|B(z)\|}{|z-k|} 
\leq 
4 \sum_{|n|>N_0}
\sigma_{|n|} \sup_{z \in \partial \Pi_n} \frac1{|z-k|}
\leq
8 \sigma_{N_0} + 
8 \sum_{|n|>N_0, n \neq k}
\frac{ \sigma_{|n|} }{|n-k|}.
\end{equation}
To estimate the resulting sum, we employ \cite[Lemma~2]{Mityagin-2016-106}, namely, 
\begin{equation}\label{MS.lem.HT}
\sum_{m=1, k \neq m}^\infty \frac{1}{m^\gamma |m-k|} 
= 
\begin{cases}
\BigO \left(\frac{\log k}{k^\gamma}\right), & \text{if } 0 <\gamma \leq 1,
\\[1mm]
\BigO \left( \frac1 k \right), & \text{if } \gamma > 1,
\end{cases}
\quad k \to + \infty;
\end{equation}
and the easier estimate for $k \to - \infty$
\begin{equation}\label{MS.lem.HT.2}
\sum_{m=1, k \neq m}^\infty \frac{1}{m^\gamma |m-k|} 
	= 
	\begin{cases}
		\BigO \left( \frac{\log k}{k} \right), & \text{if } \gamma = 1,
		\\[1mm]
		\BigO \left( \frac1 {k^\gamma} \right), & \text{if } 0<\gamma \neq  1,
	\end{cases}
	\quad k \to  - \infty.
\end{equation}
Recalling \eqref{sig.dec}, we obtain that there exists $M \equiv M(\alpha,\eps)>0$ such that for all $|k|>1$
\begin{equation}
\sum_{|n|>N_0, n \neq k}
\frac{ \sigma_{|n|} }{|n-k|}  
\leq M \log |k|
\begin{cases}
|k|^{-\frac \alpha 2 + \eps}, & \text{if } 0< \alpha \leq \frac 12,
\\ 
|k|^{-\frac 18}, & \text{if } \alpha > \frac 12,
\end{cases}
\end{equation}

Returning to \eqref{KVK.cond.Dirac} and \eqref{sum.conv.2}, we obtain that the condition \eqref{sum.conv} is satisfied (hence also \eqref{RB.sum.crit}) and so the Riesz property of the system \eqref{S0Pk.syst} follows. 
\qed

{
	\bibliographystyle{halpha}
	\bibliography{references}
}

\end{document}